\documentclass[a4paper,oneside,reqno]{amsart}

\usepackage[english]{babel}
\usepackage[utf8]{inputenc}		
\usepackage[scaled]{helvet}		
\usepackage{courier}			
\usepackage{eulervm}			
\usepackage[bb=boondox,bbscaled=1.05,scr=dutchcal]{mathalfa}	%
\usepackage{dsfont}
\normalfont

\usepackage{indentfirst}				
\usepackage{tabularx,booktabs}			
\usepackage{caption,subcaption}			
\usepackage{fullpage}				    
\usepackage[english]{varioref}			
\usepackage[dvipsnames]{xcolor}			
\usepackage{hyperref}		
	\hypersetup{breaklinks=true}			
\usepackage{bookmark}					
\usepackage{textcomp}
\usepackage{rotating,afterpage,pdflscape}
\usepackage{enumerate}
\usepackage{framed}
\usepackage{nicefrac}
\usepackage[capitalise,noabbrev]{cleveref}
	\crefname{equation}{equation}{equations}

\usepackage[textsize=footnotesize,textwidth=0.85in]{todonotes}
\presetkeys%
    {todonotes}%
    {color=Dandelion}{}%
\definecolor{webbrown}{rgb}{0.65, 0.16, 0.16}

\hypersetup{
colorlinks=true, linktocpage=true, pdfstartpage=1, pdfstartview=FitV,breaklinks=true, pdfpagemode=UseNone, pageanchor=true, pdfpagemode=UseOutlines,plainpages=false, bookmarksnumbered, bookmarksopen=true, bookmarksopenlevel=1,hypertexnames=true, pdfhighlight=/O,urlcolor=webbrown, linkcolor=RoyalBlue, citecolor=ForestGreen}

\usepackage{graphicx}		
\usepackage{float}
\usepackage{tikz}			
\usepackage{tikz-cd}		
\usetikzlibrary{
  arrows,decorations.pathreplacing,decorations.markings,shapes.geometric,positioning,calc,fadings,angles
	}
\tikzfading[name=inner fade, inner color=transparent!0, outer color=transparent!100]
\graphicspath{{Images/}}

\usepackage[textsize=footnotesize,textwidth=0.85in]{todonotes}
\presetkeys%
    {todonotes}%
    {color=Dandelion}{}%
\usepackage{amsmath,amssymb,amsthm,mathtools}
\usepackage{bm,braket,faktor,stmaryrd}

\numberwithin{equation}{section}

\newcommand{\ZZ}{\mathbb{Z}}

\newcommand{\MM}{\overline{\mathcal{M}}}
\newcommand{\mc}{\mathcal}

\newcommand{\ch}{{\rm ch}}

\DeclareMathOperator{\Aut}{Aut}
\DeclareMathOperator{\maxdeg}{maxdeg}
\DeclareMathOperator{\br}{br}
\DeclareMathOperator{\rank}{rank}

\theoremstyle{plain}
\newtheorem{thm}{Theorem}[section]

\newtheorem{prop}[thm]{Proposition}

\newtheorem{lem}[thm]{Lemma}
\newtheorem{cor}[thm]{Corollary}

\theoremstyle{definition}

\newtheorem{rem}[thm]{Remark}

\usepackage[citestyle=numeric, minbibnames=5, maxbibnames=99, giveninits=true, isbn=false, backend=biber, doi=false, url=false]{biblatex} 
\renewbibmacro{in:}{}
\usepackage{csquotes}
\bibliography{gjvconjecture.bib}

\title{On some hyperelliptic Hurwitz-Hodge integrals}
	
\vspace{.2cm}

\author{Danilo Lewa\'{n}ski }
\address{D.~L.: Universit\'e de Gen\`eve, Section de Math\'ematiques, rue du Conseil-G\'en\'eral 7-9, 1205 Gen\`eve \\
\emph{\&} Institut des Hautes \'Etudes Scientifiques, le Bois-Marie, 35 route de Chartres, 91440 Bures-sur-Yvette, France \\
\emph{\&} Universit\'e Paris-Saclay, CNRS, CEA, Institut de physique th\'eorique, 91191 Gif-sur-Yvette, France.
}
\email{danilo.lewanski@ihes.fr}
	
\subjclass[2010]{14N10, 14H10, 05A15}

\date{}

\begin{document}

\begin{abstract}
This short note addresses Hodge integrals over the hyperelliptic locus. Recently Afandi computed, via localisation techniques, such one-descendant integrals and showed that they are Stirling numbers. We give another proof of the same statement by a very short argument, exploiting Chern classes of spin structures and relations arising from Topological Recursion in the sense of Eynard and Orantin. 

These techniques seem also suitable to deal with three orthogonal generalisations: 1. the extension to the $r$-hyperelliptic locus, 2. the extension to an arbitrary number of non-Weierstrass pairs of points, 3. the extension to multiple descendants.
\end{abstract}

\maketitle



%
%

\vspace{-0.5cm}
\section{Introduction}
\label{sec:Intro}

Moduli spaces of curves have been proved in the last decades to be of great interest not only for pure Algebraic Geometry, but a key element at the crossroad of Gromov-Witten theory, Integrable Systems, Random Matrix Models, Topological Recursion, and more.\par

An important task is the computation of intersection numbers. To fix the ideas, one could think of intersection numbers as integrals packed in polynomials $P$ of the following form:
\begin{equation}
P(x_1, \dots, x_n) = \int_{\mathcal{M}(g,n)} \frac{\mathcal{C}}{\prod_j (1 - \psi_j x_j)} \cdot [\Delta] \in \mathbb{Q}[x_1, \dots, x_n], \qquad \qquad \deg P \leq \dim \mathcal{M}(g,n).
\end{equation}
Here $\mathcal{M}(g,n)$ is some moduli space of curves 
\footnote{typically the Deligne-Mumford moduli space of stable curves $\overline{\mathcal{M}}_{g,n}$, but also moduli spaces with enriched structure, such as the space of $r$-spin curves, of admissible curves, of stable maps to a specific target. Other examples of compactifications for smooth curves are provided by Hassett moduli spaces, which specialise both to the Deligne-Mumford and to the Losev-Manin compactification.}
, either compact or with a virtual fundamental class, $\mathcal{C}$ is a cohomology class of $\mathcal{M}(g,n)$ 
\footnote{typically Cohomological Field Theories (CohFTs) \cite{KMC94} or top Chern classes of those, or partial CohFTs, $F$-CohFTs, and so on. Few examples are: Hodge classes, $\Omega$-classes, Witten classes, double ramification cycle, classes of holomorphic differentials.}
, $[\Delta]$ is a locus of curves of $\mathcal{M}(g,n)$
\footnote{
for instance the locus of hyperelliptic curves and their generalisations, or partial compactification of the space $\mathcal{M}(g,n)$, such as the locus of rational tails or of compact type curves inside the Deligne-Mumford compactification. 
}
, $\psi$-classes are building blocks of $H^*(\mathcal{M}(g,n))$ of complex degree one, and $x_i$ are formal variables. \par


This short note focuses on the following intersection numbers, expressable in three different but equivalent ways. We refer to \cite{Wise, Renzo, Adam2} and references therein for an exhaustive state-of-the-art.

\begin{enumerate}
\item In terms of the hyperelliptic locus $\overline{\mathcal{H}}_{g, 2g + 2, a} \subseteq \overline{\mathcal{M}}_{g, 2g + 2}$ of algebraic curves admitting a degree $2$ map to $\mathbb{P}^1$, with $2g+2$ marked Weierstrass points and $a$ pairs conjugated by the involution, and $\lambda_i = c_i(\mathbb{E})$ the $i$-th Chern class of the Hodge bundle (see \cite{Adam} for more details):
\begin{equation}\label{eq:moduli1}
\int_{\overline{\mathcal{H}}_{g, 2g + 2, 2^{a}}} \lambda_i \br^*\left( \prod_{j=1}^{2g + 2+a} \frac{1}{(1 - \psi_j x_j)} \right) ,
\end{equation}
where $\\br$ is the map associating to each hyperelliptic curve its target.
\item In terms of admissible covers and the Hodge class (see \cite{JPT} for more details):
\begin{equation}\label{eq:moduli2}
\int_{\overline{\mathcal{M}}_{0, \emptyset - (1^{2g + 2}, 0^{a)}}(\mathbb{B}\mathbb{Z}_r)} \frac{\lambda_i,}{\prod_{j=1}^{2g + 2+a} (1 - x_j\psi_j) } , \qquad \qquad \text{ for } r=2.
\end{equation}
\item In terms of the moduli space of stable curves and the $\Omega$-CohFT (see \cite{LPSZ} and Section \ref{sec:Omega:classes}):
\begin{equation}\label{eq:moduli3}
\int_{\overline{\mathcal{M}}_{0, 2g + 2 + a}}  \frac{[\deg_{H^{*}} = i].\Omega(r,0; \overbrace{1, \dots, 1}^{2g + 2+a}, \overbrace{0, \dots, 0}^{a})}{\prod_{j=1}^{2g + 2} (1 - \psi_j x_j)}, \qquad \qquad \text{ for } r=2.
\end{equation}
\end{enumerate}
The equivalence between the first and the second is well-known, the equivalence between the second and the third, for arbitrary $r$, is achieved in \cite{LPSZ}. This restatement was then employed to address a problem by Goulden-Jackson-Vakil over the existence of an ELSV formula for double Hurwitz numbers with one total ramification~\cite{DL}. A really useful Sage package has recently been developed to perform intersection theoretic calculations otherwise by far out of reach with traditional methods \cite{admcycles}.

\subsection{Results}

Progress on the evaluation of these integrals was recently made in the following formula:

\begin{thm}[\cite{Adam}] \label{thm:Afandi} Linear-Hodge one-descendant integrals over the hyperelliptic locus evaluate to:
\begin{align*}
& \int_{\overline{\mathcal{H}}_{g, 2g + 2}}\br^*\left( \psi_1^{2g - 1 - i} \right) \lambda_i =  \frac{1}{2} \cdot e_i\left(0 + \nicefrac{1}{2}, 1 + \nicefrac{1}{2}, \ldots, g-1 + \nicefrac{1}{2} \right), \\
& \int_{\overline{\mathcal{H}}_{g, 2g + 2, 2}}\br^*\left( \psi_{2g + 3}^{2g - i} \right) \lambda_i =  \frac{1}{2} \cdot e_i(1, 2, \ldots, g).
\end{align*}
Here $ e_i(x_1, \ldots, x_n)$ is the $i^{th}$ elementary symmetric polynomial.\footnote{These evaluations are in fact half of $(-1)^{g} \mathsf{s}\bigl( g, g - i, \nicefrac{1}{2} \bigr)$ and of $(-1)^{g+1} \mathsf{s}\bigl( g+1, g+1 - i, 0 \bigr)$ respectively, the generalised Stirling numbers of the first kind.}
\end{thm}

The theorem is achieved via virtual localisation techniques. We provide a very short new proof in Section \ref{sec:proofs}, exploiting $\Omega$-classes and certain vanishings arising from the theory of Topological Recursion in the sense of Eynard and Orantin, introduced in Section \ref{sec:Omega:classes}.

This result opens up natural questions about its generalisations in at least three 'orthogonal' directions:
\begin{enumerate}
\item[\textbf{Q1}] What happens to the integrals over hyperelliptic loci with an arbitrary number of non-Weierstrass pairs of points $\overline{\mathcal{H}}_{g, 2g + 2, 2^a}$?
\item[\textbf{Q2}] The hyperelliptic locus ${\overline{\mathcal{H}}_{g, 2g + 2}}$ naturally corresponds to $r$-admissible covers to the classifying space $B\mathbb{Z}_r$ for $r=2$. How do these integrals behave for higher $r$?
\item[\textbf{Q3}] How does the theory with multiple descendants behave?
\end{enumerate}

We address all three questions with the same techniques with which we give a new proof of the theorem above. We provide and prove answers in certain regimes.\\


The answer to \textbf{Q1} is readily obtained. 

\begin{prop}\label{prop:nonWeierstrass} The integrals do not depend on the additional number $a$ of non-Weierstrass pairs marked:
\begin{align*}
& \int_{\overline{\mathcal{H}}_{g, 2g + 2, 2^a}}\br^*\left( \psi_1^{2g - 1 - i + a} \right) \lambda_i =  \frac{1}{2} \cdot e_i\left(0 + \nicefrac{1}{2}, 1 + \nicefrac{1}{2}, \ldots, g-1 + \nicefrac{1}{2} \right), \\
& \int_{\overline{\mathcal{H}}_{g, 2g + 2, 2^{a+1}}}\br^*\left( \psi_{2g + 3}^{2g - i + a} \right) \lambda_i =  \frac{1}{2} \cdot e_i(1, 2, \ldots, g).
\end{align*}
\end{prop}

The answer to \textbf{Q2} is obtained under certain restrictions on the parametrisations, which we show to be in general sharp.

\begin{thm} \label{prop:higherronedescendant} For arbitrary $r$ and one descendant, if
\begin{equation}
\max \limits_{i \neq j} (b_i + b_j) \leq r, 
\end{equation}
then
\begin{equation}
\int_{\MM_{0,1 + \ell}}\!\!\!\! [\deg_{H^*} = i].\Omega_{0,n}^{[-1]}(r,0; -\overline{\mu_1}, b_1, \dots, b_{\ell}) \psi_1^{\ell - 2 -j} = \frac{1}{r} e_i\left( \frac{\langle \mu \rangle}{r}, \frac{\langle \mu \rangle}{r} + 1, \dots, \frac{\langle \mu \rangle}{r} + [b] - 1 \right)
\end{equation}
Here $\sum_j b_j = b = [b]r + \langle b \rangle$, with $0 \leq \langle b \rangle \leq r-1$ by Euclidean division, similarly for $\mu$. Moreover, $-\overline{\mu_i}$ is the only representative of $-\mu_i$ modulo $r$ between zero and $r-1$.
\footnote{
Again, the \textsc{RHS} can be given in terms of Stirling numbers of the first kind as $\nicefrac{1}{r}(-1)^{[b]} \mathsf{s}\bigl( [b], [b] - i, \nicefrac{\langle \mu \rangle}{r} \bigr)$.
}
\end{thm}

The answer to \textbf{Q3} can be given for both generalisations at the same time \textemdash \; that is, for arbitrarily many descendants and for arbitrary $r$ \textemdash\; but imposing $b=\emptyset$. The formula to compute these integrals is achieved simply by chaining together several existing results, and then simplifying the outcome by the genus-zero restriction. We explain and combine these results in Section \ref{sec:proofs}. Meanwhile we give here a temporary statement. 

\begin{prop}\label{prop:Hurwitz} Let $2g - 2 + n >0$ and let $r$ be a positive integer. Fix remainder classes 
$$
-\overline{\mu_1}, \dots, -\overline{\mu_n} \in \{0, \dots, r-1\}.
$$ 
Then for $\mu_i = \langle \mu_i \rangle + r[\mu_i]$ with $\langle \mu_i \rangle = r - (-\overline{\mu_i}) $ we have:
\begin{equation}
\int_{\MM_{0,n}}\frac{\Omega_{0,n}(r,0; -\overline{\mu_1}, \dots, -\overline{\mu_n})}{\prod_{i}(1 - \frac{\mu_i}{r}\psi_i)} 
=  \left(\prod_{i=1}^{n} \frac{\left(\frac{\mu_i}{r}\right)^{[\mu_i]}}{[\mu_i]!} \right)^{-1} \!\! \cdot c
\sum_{P \in CP}^{\textup{finite}} 
\prod_{l=1}^{\textup{finite}} (|I_{l}^{P}||L_{l}^{P}| - |J_{l}^{P}||K_{l}^{P}|).
\end{equation}
Here $c$ is an explicit combinatorial prefactor, and $|I_{l}^{P}| := \sum \mu_{i_l}$ for $i_l$ belonging to the subset $I^{P}_l \subset \{1, \dots, n\}$.\\
Moreover, there exists a polynomial $P_{\overline{\mu}}(x_1, \dots, x_n)$ of degree $n-3$ depending on the remainder classes $\overline{\mu_i}$ such that the evaluation $P_{\overline{\mu}}(\mu_1, \dots, \mu_n)$ recovers the expression above.
\end{prop} 

\subsection{Plan of the paper}
In the next section we provide the necessary background on $\Omega$-classes and in Section \ref{sec:proofs} we prove the four statements above. In Section \ref{sec:examples} we provide examples and counterexamples of the statements above.

\subsection*{Acknowledgements}

The author is supported by the SNSF Ambizione Grant ``Resurgent topological recursion, enumerative geometry and integrable hierarchies
'' hosted at the Section de Math\'{e}matique de l'Universit\'{e} de Gen\`{e}ve. 
This work is partly a result of the ERC-SyG project, Recursive and Exact New Quantum Theory (``ReNewQuantum'') which received funding from the European Research Council (ERC) under the European Union’s Horizon 2020 research and innovation programme under grant agreement No 810573, hosted, in the case of the author, at the Institut de Physique Th\`{e}orique Paris (IPhT), CEA, Universit\'{e} de Saclay, and at the Institut des Hautes \'Etudes Scientifiques (IHES), Universit\'{e} de Saclay. The author moreover thanks Adam Afandi and Alessandro Giacchetto for useful discussions, Johannes Schmitt for \textsc{admcycles} package support, and the INdAM group GNSAGA for support.

\vspace{1cm}

\section{Background}
\label{sec:Omega:classes}

In \cite{Mum83}, Mumford derived a formula for the Chern character of the Hodge bundle on the moduli space of curves $\overline{\mathcal{M}}_{g,n}$ in terms of tautological classes and Bernoulli numbers. Such class appears in the celebrated ELSV formula \cite{ELSV}, named after its four authors Ekedahl, Lando, Shapiro, Vainshtein, that is an equality between simple Hurwitz number and an integral over the moduli space of stable curves.

\medskip

A generalisation of Mumford's formula was computed in \cite{Chiodo}. The moduli space $\overline{\mathcal{M}}_{g,n}$ is substituted by the proper moduli stack $\overline{\mathcal{M}}_{g;a}^{r,s}$ of $r$-th roots of the line bundle
\begin{equation}
	\omega_{\log}^{\otimes s}\biggl(-\sum_{i=1}^n a_i p_i \biggr),
\end{equation}
where $\omega_{\log} = \omega(\sum_i p_i)$ is the log-canonical bundle, $r$ and $s$ are integers with $r$ positive, and $a = (a_1, \ldots, a_n) \in \{ 0,\ldots,r-1 \}^n$ is an integral $n$-tuple satisfying the modular constraint
\begin{equation}
	a_1 + a_2 + \cdots + a_n \equiv (2g-2+n)s \pmod{r}.
\end{equation}
This condition guarantees the existence of a line bundle whose $r$-th tensor power is isomorphic to $\omega_{\log}^{\otimes s}(-\sum_i a_i p_i)$. Let $\pi \colon \overline{\mathcal{C}}_{g;a}^{r,s} \to \overline{\mathcal{M}}_{g;a}^{r,s}$ be the universal curve, and $\mathcal{L} \to \overline{\mathcal C}_{g;a}^{r,s}$ the universal $r$-th root. In complete analogy with the case of moduli spaces of stable curves, one can define $\psi$-classes and $\kappa$-classes. There is moreover a natural forgetful morphism
\begin{equation}
	\epsilon \colon
	\overline{\mathcal{M}}^{r,s}_{g;a}
	\longrightarrow
	\overline{\mathcal{M}}_{g,n}
\end{equation}
which forgets the line bundle, otherwise known as the spin structure. It can be turned into an unramified covering in the orbifold sense of degree $2g - 1$ by slightly modifying the structure of $\overline{\mathcal{M}}_{g,n}$, introducing an extra $\ZZ_r$ stabilizer for each node of each stable curve (see \cite{JPPZ}).

\medskip

Let $B_m(x)$ denote the $m$-th Bernoulli polynomial, that is the polynomial defined by the generating series
\begin{equation}
	\frac{te^{tx}}{ e^t - 1} = \sum_{m=0}^{\infty} B_{m}(x)\frac{t^m}{m!}.
\end{equation}
The evaluations $B_m(0) = (-1)^m B_m(1) = B_m$ recover the usual Bernoulli numbers. The following formula provides an explicit formula for the Chern characters of the derived pushforward of the universal $r$-th root $\ch_m(r,s;a) = \ch_m(R^{\bullet} \pi_{\ast}{\mathcal L})$. The formula was obtained by Mumford for $r=1$ and $s=1$ \cite{Mum83}, then generalised by Bini to arbitrary integers $s$ \cite{Bini03}, then generalised by Chiodo to arbitrary positive integers $r$.

\begin{thm}[\cite{Chiodo}]
	The Chern characters $\ch_m(r,s;a)$ of the derived pushforward of the universal $r$-th root have the following explicit expression in terms of $\psi$-classes, $\kappa$-classes, and boundary divisors:
	\begin{equation} \label{eqn:Omega:formula}
		\ch_m(r,s;a)
		=
		\frac{B_{m+1}(\tfrac{s}{r})}{(m+1)!} \kappa_m
		-
		\sum_{i=1}^n \frac{B_{m+1}(\tfrac{a_i}{r})}{(m+1)!} \psi_i^m
		+
		\frac{r}{2} \sum_{a=0}^{r-1} \frac{B_{m+1}(\tfrac{a}{r})}{(m+1)!} \, j_{a,\ast} \frac{(\psi')^m - (-\psi'')^m}{\psi' + \psi''}. 
	\end{equation}
	Here $j_a$ is the boundary morphism that represents the boundary divisor with multiplicity index $a$ at one of the two branches of the corresponding node, and $\psi',\psi''$ are the $\psi$-classes at the two branches of the node.
\end{thm}

We can then consider the family of Chern classes pushforwarded on the moduli spaces of stable curves
\begin{equation}
	\Omega_{g,n}^{[x]}(r,s;\vec{a})
	=
	\epsilon_{\ast}  \exp{\Biggl(
		\sum_{m=1}^\infty (-1)^m x^{m} (m-1)! \, \ch_m(r,s;\vec{a})
	\Biggr)}
	\in
	H^{\textup{even}}(\overline{\mathcal{M}}_{g,n}).
\end{equation}
 We will omit the variable $x$ when $x = 1$.
 
 \begin{cor} \cite{JPPZ} \label{cor:Omega:Exp}
The class $\Omega_{g,n}^{[x]}(r,s;\vec{a})$ is equal to
\begin{multline*} 
\hspace{-10pt}\sum_{\Gamma\in \mathsf{G}_{g,n}} 
\sum_{w\in \mathsf{W}_{\Gamma,r,s}}
\frac{r^{2g-1-h^1(\Gamma)}}{|\Aut(\Gamma)| }
\;
\xi_{\Gamma *}\Bigg[ \prod_{v \in V(\Gamma)} e^{-\sum\limits_{m = 1} (-1)^{m-1} x^m \frac{B_{m+1}(s/r)}{m(m+1)}\kappa_m(v)} \; \cdot 
\prod_{i=1}^n e^{\sum\limits_{m = 1}(-1)^{m-1} x^m \frac{B_{m+1}(a_i/r)}{m(m+1)} \psi^m_{i}} \cdot 
\\ \hspace{+10pt}
\prod_{\substack{e\in E(\Gamma) \\ e = (h,h')}}
\frac{1-e^{\sum\limits_{m \geq 1} (-1)^{m-1} x^m \frac{B_{m+1}(w(h)/r)}{m(m+1)} [(\psi_h)^m-(-\psi_{h'})^m]}}{\psi_h + \psi_{h'}} \Bigg]\, .
\end{multline*} 
Here $B_{m+1}(y)$ are Bernoulli polynomials, $\mathsf{G}_{g,n}$ is the finite set of stable graphs of genus $g$ with $n$ legs, $\mathsf{W}_{\Gamma,r,s}$ is the finite set of decorating the leg $i$ with $a_i$ and any other half-edge with an integer in $\{0, \dots, r-1\}$ in such a way that decorations of half-edges of the same edge ($e \in E(\Gamma)$) sum up to $r$ and locally on each vertex ($v \in V(\Gamma)$) the sum of all decorations is congruent to $(2g - 2 + n)s$ modulo $r$. 
\end{cor}

\begin{rem}\label{rem:properties}
By looking at the formula above is it easy to deduce a few properties of the classes $\Omega$, see \cite{GLN} for a more exhausive list. For instance, $\Omega^{[x]}(r,r;\vec{a}) = \Omega^{[x]}(r,0;\vec{a})$, as Bernoulli polynomials satisfy $B_{m+1}(0) = B_{m+1}(1) = B_{m+1}$. Similarly by $B_{m+1}(1 - X) = (-1)^{m+1}B_{m+1}(X)$ and Newton identities one can show that $\Omega^{[x]}(r,s;a_1, \dots, a_i + r, \dots, a_n) = \Omega^{[x]}(r,s;a_1, \dots, a_n) \cdot \left( 1 + \frac{a_i}{r}\psi_i \right)$.
\end{rem}

\subsection{Riemann--Roch for $\Omega$-classes}

The Riemann--Roch theorem for an $r$-th root $L$ of $\omega_{\log}^{\otimes s}(-\sum_i a_i p_i)$ provides the following relation:
\begin{equation}
	\frac{(2g - 2 + n)s - \sum_i a_i}{r} - g + 1 = h^0(C,L) - h^1(C,L).
\end{equation}
In some cases, i.e. for particular parametrisations of $r,s,a_i$ and for topologies $(g,n)$, it can happen that either $h^0$ or $h^1$ vanish, turning $\Omega$ into an actual total Chern class of a vector bundle, so that the Riemann-Roch equation provides the rank of this bundle.  If that happens, the Riemann--Roch equation provides a bound for the complex cohomological degree of $\Omega$:
\begin{equation}
[\deg_{H^*} = k].\Omega_{g,n}(r,s;\vec{a}) = 0, \qquad \qquad \text{ for } k > \rank(R^{\bullet} \pi_{\ast}{\mathcal L}),
\end{equation}
which are usually trivial or not depending on whether the rank $< 3g - 3 + n$. One of these instances of parametrisations is provided by the following result of Jarvis, Kimura, and Vaintrob.

\begin{thm}[\cite{JKV}, Proposition 4.4] \label{thm:JKV}
Let $g=0$, let $s=0$, let $n \geq 3$, consider $a_i$ all strictly positive except at most a single $a_j$ which can be positive, or zero, or equal to $-1$. Then the $r$-th universal root does not have any global section, that is, we have
$$
h^0 = 0.
$$
\end{thm}
Under the condition of the theorem above, the rank of $R^{\bullet} \pi_{\ast}{\mathcal L}$ equals $h^1$, and therefore one gets:
\begin{equation}\label{eq:RR_JKV}
[\deg_{H^*} = k].\Omega_{0,n}(r,s;\vec{a}) = 0, \qquad \qquad \text{ for } k > \frac{\sum_i a_i }{r} - 1.
\end{equation}

\subsection{String equation for $\Omega$-classes} 

It is known \cite{LPSZ, GLN} that if $0 \leq s \leq r$ then 
\begin{equation}
\Omega(r,s; a_1, \dots, a_{n}, a_{n+1} = s) = \pi^{*}\Omega(r,s; a_1, \dots, a_{n}).
\end{equation}
By projection formula, this implies the string equation:
\begin{equation}\label{eq:stringOmega}
\int_{\overline{\mathcal{M}}_{g, n+1}} \frac{\Omega(r,s; a_1, \dots, a_{n}, a_{n+1} = s)}{\prod_{i=1}^n (1 - x_i \psi_i)} = (x_1 + \dots + x_n) \int_{\overline{\mathcal{M}}_{g, n}} \frac{\Omega(r,s; a_1, \dots, a_{n})}{\prod_{i=1}^n (1 - x_i \psi_i)}.
\end{equation}
By remark \ref{rem:properties}, $s=r$ and $s=0$ are interchangeable.

\subsection{Relations for integrals of $\Omega$-classes arising from Topological Recursion}
Topological recursion is a universal recursive procedure that produces solutions of enumerative geometric problems (see e.g. \cite{Eyn16} for an introduction). Let us very briefly mention how this can be useful to produce relations between integrals of the $\Omega$-classes. In \cite{BDKLM}, this machinery was employed to generate double Hurwitz numbers. Although they are \textit{by definition polynomials} in some formal variables $q_1, \dots, q_r$ taking care of ramification conditions, what is generated by topological recursion are formal power series containing poles in exactly one of these variables, namely $q_r$. Polynomiality implies that the coefficient of $q_r^{-k}$ for $k>0$, which can be expressed as linear combinations of $\Omega$-classes integrals, vanishes. For more details on why $\Omega$-classes integrals appear in double Hurwitz numbers (relation known as ELSV-type formulae) see e.g. \cite{LPSZ}. These vanishing result in the following statement.

\begin{thm}[\cite{BDKLM}] \label{thm:double} Let $2g - 2 + n + \ell > 0$ and let $r$ be a positive integer, and $s=0$.
\begin{itemize}
\item Let $1 \leq \mu_1, \dots, \mu_n \leq r$, and let $\mu$ be their sum.
\item Let $1 \leq b_1, \dots, b_{\ell} \leq r-1$, and let $b$ be their sum.
\item Impose  $b \equiv \mu \; (\!\!\!\! \mod r)$. Then we can write $b = \mu + r \cdot \delta$ for some integer $\delta$.
\end{itemize}
If
\begin{equation}
\mu < b, \text{ or equivalently, } \delta > 0 \qquad \qquad \text{ (negativity), }
\end{equation}
then the following finite linear combination of $\Omega$-integrals vanishes:
\begin{equation}\label{eq:multipleBDKLMvanishing}
\sum_{t = 1}^{\ell} \frac{(-1)^{\ell - t}}{t!} 
\!\!\!
\sum_{
\substack{
\boldsymbol{\rho} \in 
(\mathscr{\tilde{P}}_{r - 1})^{k} \\ \sqcup_{\kappa} \rho^{(\kappa)} = \textbf{b}^{\vee}
}
}
\prod_{\kappa = 1}^{t} \frac{\big[\frac{r - |\rho^{(\kappa)}|}{r}\big]_{\ell(\rho^{(\kappa)}) - 1}}{|\Aut{\rho^{(\kappa)}}|} \int_{\overline{\mathcal{M}}_{g,n + t}} 
\!\!\!\!\! \frac{\Omega(r,0; - \overline{\mu_1},\dots, -  \overline{\mu_n},\dots, r - |{\rho}^{(1)}|, \dots, r - |{\rho}^{(t)}|)}{\prod_{i = 1}^n \big(1 - \frac{\mu_i}{r}\psi_i\big)} = 0.
\end{equation}
Here $\mathscr{\tilde{P}}_{r - 1}$ is the set of partitions of size at most $r-1$, the Pochhammer symbol $[x]_a := x(x-1) \cdots (x- a+1)$ stands for the descending factorial, and $\textbf{b}^{\vee} = (r - b_1, \dots, r - b_{\ell})$.\\

In particular, if the condition
\begin{equation}
\max \limits_{i \neq j}(b_i + b_j) \leq r \qquad \qquad \text{ (boundedness) }
\end{equation}
is satisfied, then all summands but a single one (the term for $t=\ell$ and for $|{\rho}^{(\kappa)}| = 1$ for $\kappa = 1, \dots, \ell$) straighforwardly disappear in the relation above. In this case we obtain:
\footnote{
Some time before \cref{thm:double}, a shadow of this statement \textemdash \; already degenerated under both the negativity and the boundedness condition \textemdash \; was derived in \cite[Theorem 2]{JPT}. The entire relation in \eqref{eq:multipleBDKLMvanishing} does not make its appearance. On the other hand, \cite{JPT}[Theorem 2] carries another sufficient condition for \eqref{eq:singleJPTvanishing} to hold: this condition is

\begin{equation} 
\mu + \delta < \ell, \; \text{ or equivalently, }\;  \mu < \frac{\sum_{j=1}^{\ell} (r - b_j)}{(r-1)} \qquad \qquad \text{ (strong negativity).}
\end{equation}
Strong negativity, as the name suggests, implies negativity. However, strong negativity is not weaker nor stronger than the combination of boundedness and negativity: one can find counterexamples of both phenomena, as well as examples of parametrisations satisfying all three conditions. In fact, for the purpose of this work strong negativity does not provide any new information: whenever strong negativity occurs, boundedness also does, therefore Theorem \ref{thm:double} suffices.
}

\begin{equation}\label{eq:singleJPTvanishing}
\int_{\overline{\mathcal{M}}_{g, n+\ell} } \frac{\Omega(r,0; -\overline{\mu_1}, \dots, -\overline{\mu_n}, b_1, \dots, b_{\ell})}{\prod_{i=1}^n (1- \frac{\mu_i}{r}\psi_i)} = 0.
\end{equation}
\end{thm}

\subsubsection{The $r=2$ case}
Let us briefly discuss the specialisation of the result above to $r=2$. 
\begin{enumerate}
\item All $b_j$ must equal to one, and therefore $b = \ell$.
\item The boundedness condition is always satisfied.
\item The negativity condition reads $\mu < \ell$.
\item The strong negativity condition reads $\mu + \delta < \ell$
\item Strong negativity and negativity are equivalent, as $\delta = (\ell - \mu)/2$.
\end{enumerate}

If (3) or equivalently (4) are satisfied, we have
\begin{equation}\label{eq:thmdoubler2}
 \int_{\overline{\mathcal{M}}_{g,n + \ell}} \frac{\Omega(2,0; -\overline{\mu_1}, \dots, -\overline{\mu_n}, 1, 1, \dots, 1)}{\prod_{i=1}^n(1 - \frac{\mu_i}{2}\psi_i)} = 0
\end{equation}
where $-\overline{\mu_i}$ in this case simplifies to the parity of $\mu_i$ (it is one if $\mu_i$ odd, else zero).

\vspace{1cm}

\section{Proofs}
\label{sec:proofs}

We are now armed to prove the statements presented in the introduction.
\subsection{New proof of Theorem \ref{thm:Afandi}}

\begin{proof}
Let us first recast Theorem \ref{thm:Afandi} as in form of \Cref{eq:moduli3}.
\begin{lem}\label{lem:recast1} Let $x$ be a formal variable. Theorem \ref{thm:Afandi} is equivalent to the following statements.
\begin{align}\label{eq:1desc_restatement1}
\int_{\overline{\mathcal{M}}_{0,2g + 2}} \frac{\Omega(2,2; \overbrace{1, 1, \dots, 1)}^{2g+2}}{(1 - \frac{x}{2}\psi_1)} &= \frac{x^{g-1}}{2^{2g}}\prod_{k=1}^{g}(x - (2k-1)), 
\\
\label{eq:1desc_restatement2}
\int_{\overline{\mathcal{M}}_{0,2g + 3}} \frac{\Omega(2,2; 0, \overbrace{1,\dots, 1}^{2g+2})}{(1 - \frac{x}{2}\psi_1)} &= \frac{x^{g}}{2^{2g+1}}\prod_{k=1}^{g}(x - 2k). 
\end{align}
\begin{proof}
Simply multiply both sides of both statements of Theorem \ref{thm:Afandi} by $(-1)^i x^i$ and sum over $i$, then use the general fact for polynomial roots $\prod_{j=1}^d(x - \alpha_j) = \sum_{j=0}^d x^j e_{d-j}(\alpha_i) (-1)^{d-j}$, and recast the obtained statements in terms of moduli space of stable curves (from integrals in \eqref{eq:moduli1} to integrals in \eqref{eq:moduli3}). This concludes the proof of the Lemma.
\end{proof}
\end{lem}

Let us now prove Equation \eqref{eq:1desc_restatement1}. The \textsc{LHS} in \eqref{eq:1desc_restatement1} is a polynomial $P(x) = c\cdot \prod_{i=1}^{2g-1} (x - \alpha_i)$, where $\alpha_i$ are the roots. The constant $c$ is immediately computed as $c = [x^{2g-1}].P(x) = 2^{-(2g - 1)} \int_{\overline{\mathcal{M}}_{0,2g + 2}} \psi_1^{2g-1} 2^{-1}= 2^{-2g}$. It remains to show that zero is a zero of $P(x)$ of order $g-1$ and that $2k-1$ is a simple zero of $P(x)$ for $k = 1, \dots, g$.
The first condition is equivalent to the fact that $\Omega$ has non-trivial cohomological degree at most $g$. This is indeed true and proved by the Riemann-Roch computation for $\Omega$ in genus zero performed in \eqref{eq:RR_JKV}: $\frac{2g+2}{2} -1 =g$. It only remains to show that $P(2k-1)=0$ for all $k=1, \dots, g$ (in fact, if so, their multiplicity cannot be greater than one by degree constraint). Explicitly, the proof boils down to the following $g$ relations:
\begin{equation}
\int_{\overline{\mathcal{M}}_{0,2g + 2}} \frac{\Omega(2,2; \overbrace{1, 1, \dots, 1)}^{2g+2}}{(1 - \frac{2k-1}{2}\psi_1)} = 0, \qquad \text{ for } k=1, \dots, g.
\end{equation}
These relations are immediately implied by \Cref{thm:double} specialised as in \Cref{eq:thmdoubler2}, then further specialised to $n=1$: the vanishing holds for positive odd $\mu_1 < 2g + 1$, or in other words, for $\mu_1 = 2k-1$ for $k=1, \dots, g$. 
\footnote{In fact more is true: the relations produced by \cref{thm:double} specialised this way are all and only the relations needed to determine $P(x)$ completely. In order terms, we have just proved that \cref{thm:Afandi} and the specialisation of \cref{thm:double} to the case $r=2, n=1,$ and $\ell$ odd, are completely equivalent statements, this way reproving \cref{thm:Afandi}.}
 \Cref{eq:1desc_restatement2} is proved similarly. 
 This concludes the proof of Theorem \ref{thm:Afandi}. \footnote{
As a curiosity, we report on a different proof for the first zero of $P(x)$ different from zero. The first vanishing is in some sense geometrically stronger than the subsequent ones. For example:
\begin{align}
P(1) &= \int_{\overline{\mathcal{M}}_{0,2g + 2}} \frac{\Omega(2,2; \overbrace{1, 1, \dots, 1)}^{2g+2}}{(1 - \frac{1}{2}\psi_1)} = \int_{\overline{\mathcal{M}}_{0,2g + 2}} \Omega(2,2; \overbrace{-1, 1, \dots, 1)}^{2g+2} = 0.
\end{align}
The first equality is by definition, the second is by first property of the $\Omega$-classes in \Cref{rem:properties}, the third is by Riemann-Roch as in \eqref{eq:RR_JKV}. The vanishing occurs by integrating the pure Chern class of a vector bundle with rank strictly smaller than the dimension of the space. The following zero does not enjoy this property as $a_1 = -3$ falls outside of the hypotheses of Jarvis, Kimura and Vaintrob in \Cref{thm:JKV}.
}
\end{proof}


\subsection{Proofs of \Cref{prop:nonWeierstrass}, \Cref{prop:higherronedescendant}, and \Cref{prop:Hurwitz}} 

\begin{proof}[Proof of Proposition \ref{prop:nonWeierstrass}] 
Again, we start by recasting the result.
\begin{lem}\label{lem:recast2} Let $x$ be a formal variable. Proposition \ref{prop:nonWeierstrass} is equivalent to the following statements.
\begin{align}\label{eq:restatement1}
\int_{\overline{\mathcal{M}}_{0,2g + 2 + a}} \frac{\Omega(2,2; \overbrace{1, \dots, 1}^{2g+2}, \overbrace{0, \dots, 0}^{a})}{(1 - \frac{x}{2} \psi_1)} &= \frac{x^{g-1+a}}{4^{g}}\prod_{k=1}^{g}(x - (2k-1)), 
\\
\int_{\overline{\mathcal{M}}_{0,2g + 3 + a}} \frac{\Omega(2,2; 0, \overbrace{1,\dots, 1}^{2g+2}, \overbrace{0, \dots, 0}^{a})}{(1 - \frac{x}{2}\psi_1)} &= \frac{x^{g+a}}{2^{2g+1}}\prod_{k=1}^{g}(x - 2k).
\end{align}
Both expressions are polynomials in $x$ of degree equal to the dimension of the moduli spaces, that is, of degree $2g-1+a$ and $2g+a$ respectively.
\begin{proof}
The proof is the same as in \Cref{lem:recast1}. This concludes the proof of the Lemma. 
\end{proof}
\end{lem}
Now it suffices to apply String \Cref{eq:stringOmega} to each of the last $a$ marked points:
\begin{equation}
\int_{\overline{\mathcal{M}}_{0,2g + 2 + a}} \frac{\Omega(2,2; \overbrace{1, \dots, 1}^{2g+2}, \overbrace{0, \dots, 0}^{a})}{(1 - \frac{x}{2} \psi_1)} = x^a \cdot \int_{\overline{\mathcal{M}}_{0,2g + 2}} \frac{\Omega(2,2; \overbrace{1, \dots, 1}^{2g+2})}{(1 - \frac{x}{2} \psi_1)}.
\end{equation}
This concludes the proof of the Proposition.
\end{proof}


\begin{proof}[Proof of \Cref{prop:higherronedescendant}] 
Consider the polynomial
\begin{equation}
P (x) := \int_{\MM_{0,1 + \ell}}\frac{\Omega_{0,n}(r,0;- \overline{\mu_1}, b_1, \dots, b_{\ell})}{(1 - \frac{x}{r}\psi_1)}  = c\cdot \prod_{i} (x - \alpha_i).
\end{equation}
It is a polynomial of degree $\deg(P) = \dim_{\mathbb{C}}(\MM_{0,1 + \ell}) = \ell - 2$, and its leading coefficient can be easily computed as 
$$c =  \int_{\MM_{0,1 + \ell}} \Omega_{0,n}(r,0; -\overline{\mu_1}, b_1, \dots, b_{\ell})\frac{\psi_1^{\ell-2}}{r^{\ell-2}} = \frac{1}{r^{\ell-2}\cdot r}  \int_{\MM_{0,1 + \ell}} \psi_1^{\ell-2} = \frac{1}{r^{\ell - 1}}.$$
 Since $\mu_1 = \mu \equiv b$ modulo $r$, we must have that $\langle \mu \rangle = \langle b \rangle$, and therefore $\bar{-\mu_1} = r - \langle \mu \rangle$. 
The lowest degree in $x$ of $P(x)$ can be computed as 
\begin{align*}
\deg(P) - \maxdeg_{H^{*}}&\Omega_{0,n}(r,0; -\overline{\mu_1}, b_1, \dots, b_{\ell})  
= \ell - 2 - \left( \frac{ r - \langle \mu \rangle + [b]r + \langle b \rangle }{r} - 1\right) = \ell - 2 - [b].
\end{align*}
So far we achieved to show that $P$ has the form
$$
P(x) = \frac{x^{\ell - 2 - [b]}}{r^{\ell - 1}}\prod_{i=1}^{[b]}(x - \alpha_i).
$$
for some suitable roots $\alpha_i$. The question is whether the TR vanishing \eqref{eq:singleJPTvanishing} can guarantee at least (and therefore all) $[b]$ roots. 
\footnote{
One could wonder whether strong negativity $\mu + \delta < \ell$ can also grant a sufficient condition to determine $P$ completely. Curiously, one finds that in this case strong negativity implies boundedness, therefore not providing anything additional. To see this, expand strong negativity as $k < \frac{\ell - \langle b \rangle - [b]}{(r-1)}$. The best possible case is given for $ \langle b \rangle = 0$, for which $P$ is determined if the first $b$ values of $\mu$ for $k \in \{0,1,\dots, [b]-1\}$ give vanishing. Substituting the highest $k = [b]-1$ one finds
$b \leq \ell + r-2$. As the parts of $b$ are at least one, write $b = \ell + |\alpha|$ for $|\alpha|$ the size of a partition of length up to $\ell$ to be distributed over the $b_i = 1 + \alpha_i$. Then $|\alpha| \leq r-2$, sharply implying boundedness.
}
Assuming boundedness, we only have to worry about negativity $\mu < b$ holding true, which means that 
$$
\langle \mu \rangle + rk < \langle b \rangle + r[b]
$$
for a few possible non-negative integers $k$. As $\langle \mu \rangle = \langle b \rangle$, negativity holds true for the values $k=0, \dots, [b]-1$, which amounts to $[b]$ different values of $\mu$ providing $[b]$ simple roots $P(\langle \mu \rangle + rk) =0$ as required. 
This concludes the proof of the Theorem.
\end{proof}



\begin{proof}[Proof of Proposition \ref{prop:Hurwitz}] 
As anticipated in the introduction, the proposition is simply obtained by a chain of several known results, simplified in genus zero.

\begin{enumerate}

\item By \cite[Section 5]{LPSZ} we have: 
$$
\int_{\MM_{g,n}}\frac{\Omega_{g,n}(r,0; -\overline{\mu_1}, \dots, -\overline{\mu_n})}{\prod_{i}(1 - \frac{\mu_i}{r}\psi_i)} = \int_{\MM_{g,\emptyset - \mu}(\mathcal{B}\mathbb{Z}_r)}\frac{\sum_{i=0} (-1)^i \lambda_i^U}{\prod_{i}(1 - \frac{\mu_i}{r}\bar{\psi}_i)}
$$
for $U$ the representation of the cyclic group $\mathbb{Z}_r$ sending $1$ to $e^{2\pi i/r}$.
\item By \cite[Theorem 1]{JPT} we have: 
$$
h_{g,\mu}^{(r), \circ}
=
 r^{2g - 2 + n + \frac{|\mu|}{r}}\left(\prod_{i=1}^{n} \frac{\left(\frac{\mu_i}{r}\right)^{[\mu_i]}}{[\mu_i]!} \right) \cdot \int_{\MM_{g,\emptyset - \mu}(\mathcal{B}\mathbb{Z}_r)}\frac{\sum_{i=0} (-1)^i \lambda_i^U}{\prod_{i}(1 - \frac{\mu_i}{r}\bar{\psi}_i)}
$$
where $h_{g,\mu}^{(r), \circ}$ are Hurwitz numbers enumerating connected genus $g$ degree $d = |\mu|$ ramified covers of the Riemann sphere with $b = b(g) = 2g - 2 + \frac{|\mu|}{r} + n$ simple ramifications, except the ramification with profile $(r,r,\dots, r)$ above zero and the ramification with profile $(\mu_1, \dots, \mu_n)$ above infinity, with $n = \ell(\mu)$. Moreover, the integral is a polynomial of degree $3g - 3 + n$ in the $\mu_i$ depending on the remainder classes $\langle \mu_i \rangle$ modulo $r$, whereas the exponential prefactor is manifestly not polynomial in the parts $\mu_i$. This property is known as quasi-polynomiality, and has been shown independently in \cite{DLPS} in the framework of Topological Recursion.

\item By Okounkov \cite{O} and Okounkov and Pandharipande \cite{OP}, we have that Hurwitz numbers can be efficiently written as vacuum expectation of operators in the Fock space, which in this case form a handy algebra closed under commutation relations:
\begin{equation}
h_{g, \mu}^{(r), \circ} = \frac{[z_1^2 \cdots z_{b(g)}^2]}{\prod \mu_i \cdot r^{d/r}}\Bigg{\langle} \mc{E}_{\mu_1}(0) \dots  \mc{E}_{\mu_n}(0) \mc{E}_{0}(z_1) \dots  \mc{E}_{0}(z_{b(g)}) \mc{E}_{-r}(0)^{d/r} \Bigg{\rangle}^{\circ},
\end{equation}
where $[x^a]f(x)$ selects the coefficient of $x^a$ in the formal power series $f(x)$, $b(g) = 2g - 2 + n + d/r$ is the Riemann--Hurwitz count of simple ramifications, and the following relations hold:
\begin{align*}
[\mathcal{E}_{a}(z), \mathcal{E}_b(w)] = 2\sinh \left(\frac{aw - bz}{2}\right) \mathcal{E}_{a+b}(z + w), 
\end{align*}
and
\begin{equation}
 \bigg{\langle} \mathcal{E}_0(z) \bigg{\rangle} = \frac{1}{2\sinh(z/2)}, \qquad \qquad \qquad \mathcal{E}_{k} \bigg{\rangle} = 0 =  \bigg{\langle} \mathcal{E}_{-k} , \qquad \text{ for } k >0.
\end{equation}

\item By \cite{Johnson} we have an algorithm that computes the vacuum expectation explicitly, iteratively commuting the operators $\mathcal{E}$ with negative indices from left to right until they hit the vacuum $\big \rangle$ and vanish. Along the way, they generate a large amount of summands from the commutation relation (intuitively speaking, the number of summands "doubles" at every commutation, although many terms end up vanishing at some further iteration of the algorithm). The algorithm defines a finite sum over $P$ running over the set of Commutation Patterns $CP$ (see \cite{Johnson}), obtaining
\begin{equation}
h_{g, \mu}^{(r)} = \frac{[u^{2g - 2 + n + d/r}]}{\prod \mu_i \cdot r^{d/r}}\frac{1}{2\sinh(u\cdot d/2)} \sum_{P \in CP} \prod_{l=1}^{n-1+d/r} 2\sinh \left((u/2)(|I_{l}^P||L_{l}^P| - |J_{l}^P||K_{l}^P|)\right)
\end{equation}
Here each $I^P_t$, $J^P_t$, $K^P_t$, $I^P_t$ is a sum of a certain subset of the $\mu_i$.

\item Restricting to genus zero forces to collect the minimal power of $u$, that is, to substitute all $\sinh (X)$ simply with their arguments $X$. We obtain:
\begin{equation}
h_{0, \mu}^{(r)} = \frac{1}{\prod \mu_i \cdot r^{d/r} \cdot d} \sum_{P \in CP} \prod_{l=1}^{n-1+d/r}  (|I_{l}^P||L_{l}^P| - |J_{l}^P||K_{l}^P|).
\end{equation}
\end{enumerate}
Putting everything together, one obtains
\begin{equation*}
\int_{\MM_{0,n}}  \!\!\!\!  \frac{\Omega(r,0; -\overline{\mu_1}, \dots, -\overline{\mu_n})}{\prod_{i}(1 - \frac{\mu_i}{r}\psi_i)} = \frac{\left(r^{-1}\right)^{2g - 2 + n + 2\frac{|\mu|}{r}}}{\prod \mu_i \cdot d}
\left(\prod_{i=1}^{n} \frac{\left(\frac{\mu_i}{r}\right)^{[\mu_i]}}{[\mu_i]!} \right)^{-1} \!\!\! \sum_{P \in CP} \prod_{l=1}^{n-1+d/r}  (|I_{l}^P||L_{l}^P| - |J_{l}^P||K_{l}^P|).
\end{equation*}
This concludes the proof fo the Proposition.
\end{proof}

\appendix
\section{Examples and counterexamples}
\label{sec:examples}

The following computations have been run through the \textsc{admcycles} Sage package. 
\footnote{For $g=2$ (i.e. $n=6$), we have run Lagrange interpolation in $x$ using \textit{integer values of $x$ in the right modular residue class}, and only after computing enough evaluations and interpolating one is allowed to consider the expression as an abstract polynomial, remembering its residue class dependence.}

\subsection{Theorem \ref{thm:Afandi} and Question Q1} For $r=2$ and $\mu = 1 + 2k$ we have:
$$
\frac{1}{x}\int_{\overline{\mathcal{M}}_{0,7}} \frac{\Omega^{[2]}(2,0; 1,1,1,1,1,1,0)}{(1 - x\psi_1)} = \int_{\overline{\mathcal{M}}_{0,6}} \frac{\Omega^{[2]}(2,0; 1,1,1,1,1,1)}{(1 - x\psi_1)}  = \frac 1 2 x (x-1)(x-3).
$$

Note that in general one can reabsorb several powers of $r$ on both sides of \Cref{thm:Afandi} and of \Cref{prop:nonWeierstrass}, by rescaling $x \mapsto rx$ and activating the degree parameter of the $\Omega$-classes.

\subsection{Question Q2} We have seen that for $n=1$ there is enough room for negativity to be satisfied so that enough evaluations of $\mu$ provide vanishing for $P$ to be determined. We want here to test boundedness condition. 

Let us for instance choose a prime number $r=13$, so that it does not possibly factorise with anything else. For $\langle \mu_1 \rangle = 9$ we have $-\overline{\mu_1} = 13 - 9 = 4$ and picking a vector $b = (4,3,6,2,7)$ sharply hitting boundedness ($6+7 = r$) we see a confirmation of our expectations:
\begin{align*}
\int_{\MM_{0,1+5}}\frac{\Omega(13,0; 4,4,3,6,2,7)}{(1 - \frac{x}{13}\psi_1)}  = \frac{x^2}{13^4} (x - \alpha), \qquad \qquad \alpha = 9 = \langle \mu_1 \rangle.
\end{align*}
We now wiggle a bit the vector $b$ outside boundedness (though preserving both its size and $\langle \mu_1 \rangle$), and the theorem immediately fails:
\begin{align*}
\int_{\MM_{0,1+5}}\frac{\Omega(13,0; 4,4,3,6,1,8)}{(1 - \frac{x}{13}\psi_1)}  = \frac{x^2}{13^4} (x - 8),  \qquad \qquad  \int_{\MM_{0,1+5}}\frac{\Omega(13,0; 4,1,2,9,2,7)}{(1 - \frac{x}{13}\psi_1)}  = \frac{x^2}{13^4} (x - 6).
\end{align*}

Other curious things can happen. Here we pick $r=3$ and we again exceed boundedness. For $\langle \mu_1 \rangle = 1$ and $b$ high enough to produce two non-zero roots, we find that one is expected and the other is not:
\begin{align*}
\int_{\MM_{0,1+5}}\frac{\Omega(3,0; 2,1,2,2,1,1)}{(1 - \frac{x}{3}\psi_1)}  = \frac{x^1}{3^4} (x - 1)(x-3).
\end{align*}
Here for $\langle \mu_1 \rangle = 2$ and $b$ high enough to produce two non-zero roots, we find one expected root, but with unexpected multiplicity:
\begin{align*}
\int_{\MM_{0,1+5}}\frac{\Omega(3,0; 1,2,2,2,1,1)}{(1 - \frac{x}{3}\psi_1)}  = \frac{x^1}{3^4} (x - 2)^2
\end{align*}

Also, when $\delta$ is high enough for $\mu_1 = \langle \mu_1 \rangle$, it is possible that $P$ does not even factorise in $\mathbb{R}$ anymore:
\begin{align*}
\int_{\MM_{0,1+5}}\frac{\Omega(3,0; 1,2,2,2,2,2)}{(1 - \frac{x}{3}\psi_1)}  = \frac{x^0}{3^4} (x-1) \left(x - \frac{1}{2} + i\frac{\sqrt{11}}{2}\right)\left(x -\frac{1}{2} - i \frac{\sqrt{11}}{2}\right).
\end{align*}

\vspace{2cm}
\printbibliography

\end{document}